\documentclass[11pt,leqno,twoside]{article}
\usepackage[english]{babel}
\usepackage{mathrsfs,amssymb,amsfonts,amsthm,amsmath,amsbsy}
\usepackage{dsfont,verbatim,fancyhdr}
\usepackage{natbib} 
\usepackage{enumerate}
\usepackage{marginnote}
\usepackage{graphicx,float}
\usepackage{cancel}
\usepackage[hang]{caption}
\usepackage[usenames,dvipsnames]{xcolor}
\usepackage[breaklinks=true,bookmarks,bookmarksnumbered,colorlinks=true,pdfstartview=FitV, linkcolor=Red,citecolor=blue!90!green,urlcolor=blue!90!green]{hyperref}
\usepackage{tkz-graph}
\usepgflibrary{arrows}
\usepackage{bm}
\usepackage[normalem]{ulem}

\allowdisplaybreaks
\bibpunct{\textcolor{blue!90!green}{(}}{\textcolor{blue!90!green}{)}}{\textcolor{blue!90!green}{;}}{a}{\textcolor{blue!90!green}{,}}{\textcolor{blue!90!green}{;}}

\newcommand{\bs}[1]{\boldsymbol{#1}}

\definecolor{iblue}{rgb}{0.1,0,0.75}
\definecolor{ired}{rgb}{0.9,0,0.1}

\newcommand{\norm}[1]{|{#1}|}

\def\R{\mathbb{R}}
\def\Z{\mathbb{Z}}
\def\N{\mathbb{N}}
\def\P{\mathbb{P}}
\def\X{\mathbb{X}}

\def\d{\mathrm{d}}

\def\ua{\uparrow}

\newcommand{\Po}{\mathrm{Po}}
\newcommand{\E}{\mathbb{E}}
\newcommand{\ddr}{\mathrm{d}}
\newcommand{\edr}{\mathrm{e}}
\newcommand{\MLud}{M_{L}} 
\newcommand{\MLc}{\tilde M_{L}} 
\newcommand{\Mud}{M} 
\newcommand{\Ltrans}{p_{L}}
\newcommand{\bdtrans}{p}
\newcommand{\Lstat}{\pi_{L}}

\newcommand{\Indic}{\mathds{1}}

\newtheorem{theorem}{Theorem}[section]

\newtheorem{proposition}[theorem]{Proposition} 

\newtheorem{lemma}[theorem]{Lemma} 
 
\newtheorem{definition1}[theorem]{Definition} 
 
\newtheorem{remark1}[theorem]{Remark} 

\usepackage{ifthen}

\long\def\symbolfootnote[#1]#2{\begingroup\def\thefootnote{\hspace*{-1mm}\fnsymbol{footnote}}\footnote[#1]{#2}\endgroup}

\title{\bf \vspace{-2.5cm}
Birth-and-death P\'olya urns and\\ stationary random partitions
\\[5mm]
}
\author{
\textsc{Pierpaolo De Blasi}\\[1mm]
\emph{University of Torino and Collegio Carlo Alberto}\\[3mm]
\textsc{Matteo Ruggiero\footnote{Email: matteo.ruggiero@unito.it}}\\[1mm]
\emph{University of Torino and Collegio Carlo Alberto}\\[3mm]
\textsc{Stephen G. Walker}\\[1mm]
\emph{University of Texas at Austin}
}
\date{\today}

\fancypagestyle{plain}{
\fancyhead{}
}

\begin{document}
\thispagestyle{empty}

\maketitle

\begin{center}
\begin{minipage}{.75\textwidth}
\footnotesize\noindent
We introduce a class of birth-and-death P\'olya urns, which allow for both sampling and removal of observations governed by an auxiliary inhomogeneous Bernoulli process, and investigate the asymptotic behaviour of the induced allelic partitions. By exploiting some embedded models, we show that the asymptotic regimes exhibit a phase transition from partitions with almost surely infinitely many blocks and independent counts, to stationary partitions with a random number of blocks. The first regime corresponds to limits of Ewens-type partitions and includes a result of \citet{ABT92} as a special case. 
We identify the invariant {and reversible} measure in the second regime, which preserves asymptotically the dependence between counts,
and is shown to be a mixture of Ewens sampling formulas, with a tilted Negative Binomial mixing distribution on the sample size.
\\[-2mm]


\end{minipage}
\end{center}
\linespread{1.2}


\section{Introduction and outline of the results}\label{subsec: motivation}

P\'olya urn schemes provide easily interpretable exchangeable sequences and are among the most celebrated sampling rules in probability. See \citet{JK77} and \citet{M09} for general treatments. Of particular interest for our purposes is the Blackwell--MacQueen P\'olya urn (\citealp{BM73}): given $\lambda>0$ and a nonatomic probability measure $P_{0}$ on a Polish space $\X$, a sequence sampled from a P\'olya urn is such that $X_{1}\sim P_{0}$ and for $n\ge1$
\begin{equation}\label{Polya urn predictive}
X_{n+1}\mid X_{1},\ldots,X_{n}\sim \frac{\lambda}{\lambda+n}P_{0}
  +\frac{1}{\lambda+n}\sum_{i=1}^{n}\delta_{X_{i}},
\end{equation} 
where $\delta_{x}$ denotes a point mass at $x$. Since $P_{0}$ has no atoms, if $X_{n+1}$ is sampled from $P_{0}$ a new value is observed, otherwise $X_{n+1}$ is a copy of a previous observation. Hence a P\'olya urn sample will feature ties with positive probability, inducing a partition of the observed values. A popular interpretation of the above scheme is as a species sampling model (\citealp{P96}), whereby the observations label species sampled from a large population, and those drawn from $P_{0}$ are species that have not been previously observed. 

The impact of the Blackwell--MacQueen P\'olya urn schemes and its developments has been extremely significant in applied probability and statistics, particularly through the construction and characterisation of random probability measures via limits of exchangeable sequences \citep{BM73,P95,P96,P06,GP05,LMP05,LMP07}, and as a building block in the architecture of computational strategies for posterior inference with Bayesian nonparametric mixture models \citep{EW95,MM98,N00,IJ01}. 


In these respects, particularly relevant is its relationship with the Ewens sampling formula, which assigns probability
\begin{equation}\label{ESF}
\text{ESF}_{n}(m_1,\ldots,m_n)
  =\frac{n!\Gamma(\lambda)}{\Gamma(\lambda+n)}
  \prod_{i=1}^n\left( \frac{\lambda}{i} \right)^{m_i}
  \frac{1}{m_i!}
{\Indic\Big\{ \sum\nolimits_{j=1}^n jm_j=n \Big\}},
\end{equation}   
to vectors $m=(m_{1},\ldots,m_{n})\in \Z_{+}^{n}$, 
{where $\Indic(\,\cdot\,)$ is the indicator function}.
Originally introduced for describing the sampling distribution of allelic frequencies in a neutral population at equilibrium (\citealp{E72}), this provides the law of the ``allelic'' partition $(m_1,\ldots,m_n)$ induced by a P\'olya urn sample of size $n$, where $m_{i}$ is the number of alleles appearing exactly $i$ times. Equivalently, it provides the law of the partition induced by sampling from a Dirichlet process random probability measure (\citealp{A74}).
See \citet{C16} (with discussion) for a recent review of applications and connections of the Ewens sampling formula. 

In this paper we consider a class of birth-and-death P\'olya urns (B\&D-PUs for short), which in addition to adding observations according to \eqref{Polya urn predictive} allow to remove observations from the current sample, and investigate their asymptotic regimes under a certain specification of the probability of 
{a removal step}.
Rather than extending the predictive distribution in \eqref{Polya urn predictive}, we define these directly in terms of the dynamics induced on the associated allelic partition. For any $\beta\in(0,1]$, define a B\&D-PU 
as the Markov chain $\Mud=\{\Mud(h),h\in\Z_{+}\}$ with state space $\Z_{+}^{\infty}$ and
transition probabilities
\begin{equation}\label{BD-PU transitions}
  \bdtrans(m'|m)\propto
  \begin{cases}
 \beta\lambda, \quad   & m'=m+e_{1},\\
\beta im_{i},
  & m'=m-e_{i}+e_{i+1},\quad i\ge1  \\
(1-\beta) im_{i},   & m'=m-e_{i}+e_{i-1},\quad i\ge1
  \end{cases}
\end{equation} 
where $\propto$ denotes proportionality and
\begin{equation*}
  m\pm e_{i}
  =(m_{1},\ldots,m_{i-1},m_{i}\pm 1,m_{i+1},\ldots),
  \quad\quad m\pm e_{0}=m.
\end{equation*}  
{Note that the normalising constant in \eqref{BD-PU transitions} is $\beta\lambda+\sum_{i\ge1}im_{i}$}. Without loss of generality, we let for convenience $\Mud$ and other auxiliary chains introduced later start from the origin $(0,0,\ldots)$, instead of assigning them an initial distribution. 
The transitions in \eqref{BD-PU transitions} correspond respectively  to the introduction of a new block of size one; to a block of size $i$ becoming of size $i+1$; and to a block of size $i$ becoming of size $i-1$. Here $im_i$ is the total number of items in blocks of size $i$. The dynamics of $\Mud$, 
restricted to its first two coordinates, are depicted on a lattice in Figure \ref{fig: graph}. 
Equivalently, the above transitions can be expressed in terms of the underlying process of observations, whereby with probability 
\begin{equation}\label{bernoulli parameter}
  b(\beta,m)=\frac{\beta(\lambda+\sum_{i\geq1}im_i)}
  {\beta\lambda+\sum_{i\geq 1}im_i}
\end{equation}  
a further observation is drawn from \eqref{Polya urn predictive} and added to the sample, and with probability $1-b(\beta,m)$ an observation is chosen uniformly from the current sample and removed. 
Setting $\beta=1$ above reduces \eqref{BD-PU transitions} to the usual dynamics induced on partitions by the P\'olya urn \eqref{Polya urn predictive} (cf., e.g., \citet{F10}, Section 2.7.2), whereas $\beta=0$ would simply remove sequentially all items currently available until none is left, hence it is not considered here. Note that the random allelic partitions induced by P\'olya urns are consistent under uniform deletion, i.e., the partition obtained by removing a uniformly chosen item from an $\text{ESF}_{n}$-distributed partition of $n$ elements has distribution $\text{ESF}_{n-1}$; cf., e.g.,  \citet{C16}. Hence perturbing the P\'olya urn dynamics by a finite number of uniform removals does not effect its limiting behaviour. Here, however, we are allowing for an infinite number of removals according to an auxiliary inhomogeneous Bernoulli process with state-dependent probability $1-b(\beta,m)$, and study the implied long run behaviour. 

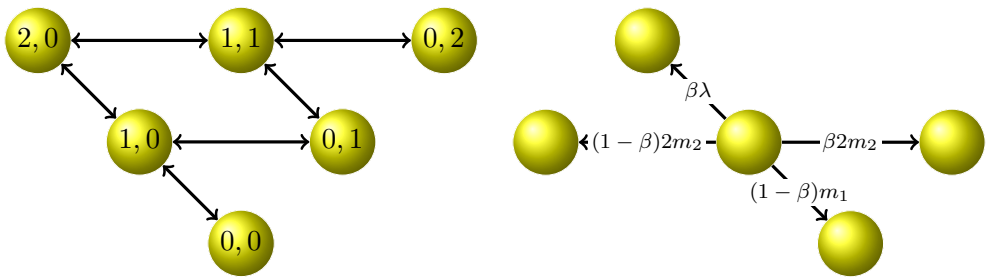
\begin{figure}[t!]
\begin{center}
\begin{tikzpicture}[scale=1.35]
 \SetGraphUnit{1}
 \GraphInit[vstyle=Shade]
 \SetGraphShadeColor{yellow}{black}{yellow}
 \tikzset{LabelStyle/.style= {draw,
                              fill  = white,
                              text  = black}}
\tikzset{VertexStyle/.style = {shape = circle,
                                shading = ball,
                                ball color = yellow,
                                minimum size = 25pt,
                                inner sep = 2.5pt,
                                draw,
			       color=white}}
\SetGraphUnit{1}
\Vertex[L=\text{\textcolor{black}{$0,0$}},x=0,y=0]{00}
\Vertex[L=\text{\textcolor{black}{$1,0$}},x=-1,y=1]{10}
\Vertex[L=\text{\textcolor{black}{$2,0$}},x=-2,y=2]{20}
\Vertex[L=\text{\textcolor{black}{$0,1$}},x=1,y=1]{01}
\Vertex[L=\text{\textcolor{black}{$0,2$}},x=2,y=2]{02}
\Vertex[L=\text{\textcolor{black}{$1,1$}},x=0,y=2]{11}
\tikzset{EdgeStyle/.style = {<->, very thick}}
\Edge(00)(10)
\Edge(10)(20)
\Edge(10)(01)
\Edge(20)(11)
\Edge(11)(02)
\Edge(01)(11)
\tikzset{VertexStyle/.style = {}}
 \end{tikzpicture}
 \hspace{2mm}
 \begin{tikzpicture}[scale=1.35]
 \SetGraphUnit{2}
 \GraphInit[vstyle=Shade]
 \SetGraphShadeColor{yellow}{red}{yellow}
 \tikzset{LabelStyle/.style= {
                              fill  = white,
 				inner sep = 1pt,
                              text  = black}}
\tikzset{VertexStyle/.style = {shape = circle,
                                shading = ball,
                                ball color = yellow,
                                minimum size = 25pt,
                                inner sep = 2.5pt,
                                draw,
			       color=white}}
\SetGraphUnit{1}
\Vertex[L=\text{},x=0,y=0]{0}
\Vertex[L=\text{},x=-1,y=1]{1}
\Vertex[L=\text{},x=2,y=0]{3}
\Vertex[L=\text{},x=-2,y=0]{4}
\Vertex[L=\text{},x=1,y=-1]{6}
\tikzset{EdgeStyle/.style = {->, very thick}}
\Edge[label=\scriptsize$\beta\lambda$](0)(1)
\Edge[label=\scriptsize$\beta 2m_{2}$](0)(3)
\Edge[label=\scriptsize$(1-\beta)2m_{2}$](0)(4)
\Edge[label=\scriptsize$(1-\beta)m_{1}$](0)(6)

\tikzset{VertexStyle/.style = {}}
 \end{tikzpicture}
\begin{minipage}{.75\textwidth}\vspace{4mm}
\caption{\footnotesize Graph representation of the partition-valued process induced by B\&D-PUs restricted to $(m_1,m_2)$
(left), and probabilities of admissible transitions, up to proportionality (right).
}\label{fig: graph}
\end{minipage}
\end{center}
\end{figure}

We show that B\&D-PUs exhibit asymptotically a phase transition at $\beta=1/2$ from stationary dynamics with a random number of blocks $k=\sum_i m_i$, to nonstationary dynamics which yield almost surely infinitely-many blocks. 
To this aim, we first study the stationary properties of an auxiliary system  of finite-dimensional Markov chains on the space of partitions. These differ from other commonly used partition-valued processes, e.g., those in \citet{P09}, which are indexed by a fixed sample size, or those in \citet{C14}, which are indexed by the maximum number of blocks. Our auxiliary processes are instead indexed by a maximal allelic count, i.e., the maximum number of items allowed in each block, whereas both the number of blocks and the sample size are left free to vary. 
The ergodicity of these chains is then exploited by embedding them, at certain stopping times, in B\&D-PUs, whose asymptotic distributions coincide with the weak limits of the stationary laws of the auxiliary chains. 
Specifically, values of $\beta\in [1/2,1]$ for B\&D-PUs in the long run generate infinite structures analogous to those induced by limits of Ewens partitions. In this limit, the allelic counts $(m_1,m_2,\ldots)$ are asymptotically independent with Poisson distribution of mean $\lambda/i$, irrespective of the value of $\beta$. This result for $\beta=1$ was first proved by \citet{ABT92}.
Values of $\beta\in (0,1/2)$ are instead shown to generate stationary models on $\Z_{+}^{\infty}$ with {invariant and reversible} distribution 
\begin{equation}\label{pi in intro}
\pi(m)=  \frac{\beta\lambda+\sum_{i\ge1}im_{i}}
  {\beta\lambda
  +\beta\lambda/(1-2\beta)}
  \prod_{i\ge1}\text{Po}\left(m_{i};
  \theta_i/i\right), \quad
   \quad 
   \theta_i=\lambda\bigg(\frac{\beta}{1-\beta}\bigg)^{i},
\end{equation} 
where $\text{Po}(\cdot;\theta)$ is a Poisson probability mass function with mean $\theta$. 
Here several elements are of interest. The first is the dependence between the allelic counts $m_i$, which, contrary to the $\beta\geq 1/2$ case, is preserved in the limit. An interpretation for this dependence can be given by taking $m_i$ be independent Poisson variables of mean $\theta_i/i$, with $\theta_{i}$ as in \eqref{pi in intro}, and letting $J=j$ with probability proportional to $\theta_{j}$ for $j\ge1$ or to $\beta\lambda$ for $j=0$. Then the vector $(m_{1},\ldots,m_{j-1},m_{j}+1,m_{j+1},\ldots,)$ has distribution \eqref{pi in intro}. 
The second is the fact that the expected number of items $\E(im_{i})$ in groups of size $i$ can be easily checked to be proportional to $\theta_{i}$ and thus depends on $i$, whereas in the Ewens case $\E(im_{i})=\lambda$ in the limit. 
Recently, \citet{BU10,BUV11} studied the asymptotic behaviour of generalised sampling formulas where the counts distribution also depends on the count index $i$, obtained by replacing $\lambda$ in \eqref{ESF} with $\lambda_i$ for a sequence $(\lambda_1,\lambda_2,\ldots)$ of nonnegative reals. 
A third element of interest is the fact that the number of groups at stationarity is random and finite, with 
{distribution determined by $(\beta,\lambda)$; see \eqref{moments} below.} \citet{G10} studied a partition structure generated sequentially which yields a finite, random number of blocks.
Finally, in this stationary regime, the underlying system of particles is also stationary with invariant measure related to a mixture of P\'olya urn schemes, with a tilted Negative Binomial mixing distribution on the sample size. The latter is defined here as the total number of items in the system at a given time, i.e., $n:=\sum_{i\ge1}im_{i}$.

The paper is organised as follows. 
Section \ref{sec: up and down chains} defines the auxiliary system of Markov chains with maximal allelic counts and identifies their invariant measure. Section \ref{sec: delayed urn} proves the phase transition for the partition structures generated asymptotically by B\&D-PUs, identifies their limiting distributions and shows the reversibility for $\beta\in(0,1/2)$. Finally, Section \ref{sec: mix of ESFs} highlights a connection of our results with a mixture of Ewens sampling formulas.


\section{Chains with maximal allelic count}
\label{sec: up and down chains}

In this Section we define and study a system of finite-dimensional partition-valued Markov chains, which are instrumental for the investigation of the B\&D-PUs asymptotic regimes.
Fix $L\in \N$ and $\mu>0$, and define $\MLud=\{\MLud(h),h\in\Z_{+}\}$ to be the $\Z_{+}^{L}$-valued Markov chain with transition probabilities 
\begin{equation}\label{up-down transitions explicit}
  \Ltrans(m'|m)\propto
  \begin{cases}
\beta\lambda, \quad 
  & m'=m+e_{1},  \\
\beta im_{i}, 
  & m'=m-e_{i}+e_{i+1},\quad i=1,\ldots,L-1,  \\
\beta L m_{L}, 
  & m'=m-e_{L},  \\
(1-\beta) im_{i}, 
  & m'=m-e_{i}+e_{i-1},\quad i=1,\ldots,L,\\
(1-\beta)\mu, 
  & m'=m+e_{L}.
  \end{cases}
\end{equation} 
with normalising constant $\beta\lambda+(1-\beta)\mu+\norm{m}$, for 
$\norm{m}=\sum_{i\geq 1}im_{i}$. 
Here the difference with respect to \eqref{BD-PU transitions} is that the system has a maximal allelic count $m_L$, whereby a group of size $L$ which becomes of size $L+1$ is removed from the system, with probability proportional to $\beta Lm_{L}$, and groups of size $L$ can be inserted in the system, with probability proportional to $(1-\beta)\mu$. 

The following result identifies the invariant measure of $\MLud$.

\begin{theorem}\label{thm: up-down chain}
$\MLud$ with transition probabilities \eqref{up-down transitions explicit} has unique invariant distribution
\begin{align}\label{invariant up-down}
  \Lstat(m)
  =\frac{\beta\lambda+(1-\beta)\mu+\norm{m}}
  {\beta\lambda+(1-\beta)\mu+\sum_{i=1}^{L}\theta_{i}}
  \prod_{i=1}^{L}\Po(m_{i};\theta_{i}/i),
\end{align}
where 
\begin{equation}\label{eq: mixed thetas implicit}
  \theta_{i}=w_{i}(\beta)\lambda+(1-w_{i}(\beta))\mu,
  \quad \quad 
  w_i(\beta)=
  \begin{cases}
  \displaystyle
    \beta^i\frac{(1-\beta)^{L-i+1}-\beta^{L-i+1}}
  {(1-\beta)^{L+1}-\beta^{L+1}},
  &\beta\neq 1/2,\\[3mm]
  \displaystyle
  \frac{L-i+1}{L+1},
  &\beta=1/2.
  \end{cases}  
\end{equation} 
\end{theorem}

\begin{proof}
Let 
\begin{equation}\label{pi}
\tilde \pi_{L}(m)=\prod_{i=1}^{L}\Po(m_{i};\theta_{i}/i).
\end{equation} 
The global balance condition reads
  \begin{align}\label{mixed computation}
  \sum_{m\in\Z_{+}^{L}}&\,(\beta\lambda+(1-\beta)\mu+\norm{m})\tilde \pi_{L}(m)p_{L}(m'|m)\notag\\
  =&\,\beta\left(\tilde \pi_{L}(m'-e_{1})\lambda
  +\sum_{i=1}^{L-1}\tilde \pi_{L}(m'+e_{i}-e_{i+1})i(m'_{i}+1)
  +\tilde \pi_{L}(m'+e_{L})L(m'_{L}+1)\right)\notag\\
  &\,+(1-\beta)\left(\tilde \pi_{L}(m'+e_{1})(m'_{1}+1)
  +\sum_{i=2}^{L}\tilde \pi_{L}(m'+e_{i}-e_{i-1})i(m'_{i}+1)
  +\tilde \pi_{L}(m'-e_{L})\mu\right)\notag\\
  =&\,\tilde \pi_{L}(m')\Bigg[\beta\left(\frac{\lambda}{\theta_{1}}m'_{1}
  +\sum_{i=1}^{L-1}\frac{\theta_{i}}{i(m'_{i}+1)}
  \frac{(i+1)m'_{i+1}}{\theta_{i+1}}i(m'_{i}+1)
  +\frac{\theta_{L}}{L(m'_{L}+1)}L(m'_{L}+1)\right)\\
  &\,+(1-\beta)\left(\frac{\theta_{1}}{m'_{1}+1}(m'_{1}+1)
  +\sum_{i=2}^{L}\frac{\theta_{i}}{i(m'_{i}+1)}
  \frac{(i-1)m'_{i-1}}{\theta_{i-1}}i(m'_{i}+1)
  +\frac{Lm'_{L}}{\theta_{L}}\mu\right)\Bigg]\notag\\
  =&\,\tilde \pi_{L}(m')\Bigg[\beta\theta_{L}+(1-\beta)\theta_{1}+
  m'_{1}\left(\frac{\beta\lambda+(1-\beta)\theta_{2}}{\theta_{1}}\right)\notag\\
  &\,+\sum_{i=2}^{L-1}im'_{i}\left(
  \frac{\beta\theta_{i-1}+(1-\beta)\theta_{i+1}}{\theta_{i}}\right)
  +Lm'_{L}\left(\frac{\beta\theta_{L-1}+(1-\beta)\mu}{\theta_{L}}\right)\Bigg].\notag
\end{align}
The right hand side equals $\tilde \pi_{L}(m')(\beta\lambda+(1-\beta)\mu+\norm{m'})$
upon imposing
\begin{equation*}
\left\{
\begin{array}{l}
  \beta\lambda+(1-\beta)\theta_{2}=\theta_{1},   \\
  \beta\theta_{i-1}+(1-\beta)\theta_{i+1}=\theta_{i},\quad i=2,\ldots,L-1,   \\
  \beta\theta_{L-1}+(1-\beta)\mu=\theta_{L},  \\
  \beta\theta_{L}+(1-\beta)\theta_{1}=\beta\lambda+(1-\beta)\mu.
\end{array}
\right.
\end{equation*} 
The last equation equals the sum of the first $L$, hence $(\theta_1,\ldots,\theta_L)$ is the solution of the system of $L$ linear equations
\begin{equation*}
\left\{
\begin{array}{ll}
  \theta_{1}-(1-\beta)\theta_{2}= \beta\lambda \\
  \theta_{i}-\beta\theta_{i-1}-(1-\beta)\theta_{i+1}=0,\quad i=2,\ldots,L-1,\\
  \theta_{L}-\beta\theta_{L-1}=(1-\beta)\mu.
\end{array}
\right.
\end{equation*} 
Lemma \ref{lem: system} in the Appendix now implies that the solution of the system is given by \eqref{eq: mixed thetas implicit}, and the statement follows by dividing both sides of \eqref{mixed computation} by the normalising constant
  $$\sum_{m\in\Z_{+}^{L}}\,(\beta\lambda+(1-\beta)\mu+\norm{m})\tilde \pi_{L}(m)
  =\beta\lambda+(1-\beta)\mu+\sum_{i=1}^{L}\theta_{i}.$$
Finally, uniqueness follows from positive recurrence, which can be easily proved.
\end{proof}

An interpretation for the dependence among the counts $m_i$ at stationarity can be provided by means of an alternative representation for the invariant distribution. 
Let $m_{i}$ be mutually independent each with
$\text{Po}(m_{i}; \theta_{i}/i)$ 
distribution, 
 $\theta_i$ as in \eqref{eq: mixed thetas implicit}, and select $J=j$ with probability proportional to $\theta_{j}$ for $j=1,\ldots,L$ and to 
$\beta\lambda+(1-\beta)\mu$
for $j=0$. 
Then $m'=(m_{1},\ldots,m_{j-1},m_{j}+1,m_{j+1},\ldots,m_{L})$ has distribution \eqref{invariant up-down}. This can be easily shown by exploiting the fact that if $Z$ has $\text{Po}(z;\lambda)$ distribution, then $Z+1$ has probability mass function $\text{Po}(z;\lambda)z/\lambda$. 
An informal interpretation for the component-specific parameters $\theta_{i}$'s in \eqref{eq: mixed thetas implicit} can be provided by recalling that new groups enter the system from the left (i.e., $m\mapsto m+e_{1}$) with probability proportional to $\lambda$ and from the right (i.e., $m\mapsto m+e_{L}$) with probability proportional to $\mu$. 
It is easily checked that $w_i(\beta)$ is decreasing in $i$. Then \eqref{eq: mixed thetas implicit} expresses the fact that the effect of $\lambda$ (resp.~$\mu$) on $m_i$ is stronger for small (resp.~large) $i$.
For odd $L$, the median parameter $\theta_{(L+1)/2}$ simplifies to 
\begin{equation}\label{theta central}
  w_{(L+1)/2}(\beta)=\displaystyle \frac{\beta^{(L+1)/2}
  }{\beta^{(L+1)/2}+(1-\beta)^{(L+1)/2}}, 
  \quad \quad \beta\ne1/2,
\end{equation} 
which shows more explicitly the effect of $\beta$ on the median count and its dependence on the number of counts separating it from the extremal $m_{1}$ and $m_{L}$ (see the proof of Lemma \ref{lem: system}). 
When $\beta=1/2$, \eqref{theta central} further simplifies to $1/2$, and $\theta_{(L+1)/2}$ reduces to $(\lambda+\mu)/2$.


\section{Birth-and-death P\'olya urns}
\label{sec: delayed urn}

B\&D-PUs have been defined in the Introduction to be partition-valued Markov chains $\Mud$ with state space $\Z_{+}^{\infty}$ and transition probabilities \eqref{BD-PU transitions}. Informally, the underlying sampling process can be thought of as P\'olya urn sampling where particles are deleted at random times. 
Here we exploit the class of chains with maximal allelic count
$\MLud$, introduced in the previous section, for identifying the asymptotic regimes of $\Mud$.
The strategy is to let a sequence of chains with maximal capaciy $\{M_{L},L\in N\}$ converge to the B\&D-PU $M$ as $L\rightarrow \infty$, and then obtain the asymptotic regimes as appropriate limits of the marginal distributions of $M_{L}$. We achieve this by letting the probability of introducing $L$-sized blocks in $M_{L}$ be governed by $\mu_{L}$ instead of $\mu$ (cf.~\eqref{up-down transitions explicit}), and letting $\mu_{L}$ converge to zero appropriately fast as $L\rightarrow \infty$. 
The key intuition here is that since the probability of blocks entering the system for left and right is proportional to $\beta\lambda$ and $(1-\beta)\mu_{L}$ respectively, the expected number of items entering the system from left and right is proportional to $\beta\lambda$ and $(1-\beta)L\mu_{L}$ respectively, so the second term needs to go to zero appropriately fast, as $L\rightarrow \infty$, in order to obtain asymptotically well-defined dynamics. 

First we identify, with the following result, the weak limits of the invariant distribution of $\MLud$, determined in Theorem \ref{thm: up-down chain}, as $\mu_{L}$ goes to 0 with $L$. 

\begin{theorem}\label{prop: limit from dependence 2}
Let $\{\mu_L\}_{L\geq 1}\subset \R_{+}$ be decreasing and such that $L\mu_{L}\rightarrow 0$ as $L\rightarrow\infty$, and let $(Z_{1}^{(L)},\ldots,Z_{L}^{(L)})$ have distribution \eqref{invariant up-down}, with $\theta_i$ as in \eqref{eq: mixed thetas implicit} 
replacing $\mu$ with $\mu_L$. Then, as $L\to\infty$,
  $$(Z_{1}^{(L)},Z_{2}^{(L)},\ldots)
  \stackrel{d}{\to}
  (Z_{1},Z_{2},\ldots),$$ 
where $(Z_{1},Z_{2},\ldots)$:
\begin{enumerate}[{(}i{)}]
\item
are independent Poisson random variables with mean $\lambda/i$, if $\beta\in [1/2,1]$;
\item have joint distribution \eqref{pi in intro}, if $\beta\in(0,1/2)$.
\end{enumerate}
\end{theorem}
\begin{proof}
Let $w_i(\beta)$ be as in \eqref{eq: mixed thetas implicit}. As $L\to\infty$, it is easy to see that, for any $i$, $w_i(\beta)\to (\beta/(1-\beta))^i$ when $\beta<1/2$ and $w_i(\beta)\to 1$ when $\beta\geq 1/2$.
Also, let $\theta_i^{(L)}$ be as in \eqref{eq: mixed thetas implicit} with $\mu=\mu_L$. Since $\mu_L\rightarrow0$ as $L\rightarrow\infty$, $\theta_i^{(L)}\to (\beta/(1-\beta))^{i}\lambda$ when $\beta<1/2$ and $\theta_i^{(L)}\to\lambda$ when $\beta\geq 1/2$. Then, by using \eqref{eq:sum_thetaL} of Lemma \ref{lem: system} in the Appendix, as $L\to\infty$ we have
  $$\sum_{i=1}^L\theta_i^{(L)}
  =\begin{cases}
  \displaystyle  
  L\mu_L+\frac{\beta}{1-2\beta}(\lambda-\mu_L)
  +o(L\mu_L),&\mbox{if }\beta<1/2,\\[3mm]
  \displaystyle  
  \frac{\lambda+\mu_L}{2}L,
  &\mbox{if }\beta=1/2,\\[3mm]
  \displaystyle  
  \lambda L+\frac{1-\beta}{1-2\beta}(\lambda-\mu_L)
  +o(\mu_L),&\mbox{if }\beta>1/2.
  \end{cases}$$
Thus $L\mu_L\to0$ implies that $\sum_{i=1}^L \theta_i^{(L)}$ converges to $\lambda \beta/(1-2\beta)$ when $\beta<1/2$ since, and  diverges for $\beta\ge1/2$. 

Denote now by $\E_L$, $\tilde\E_L$ and $\tilde\E_\infty$ the expectations with respect to $\Lstat$ in \eqref{invariant up-down} and $\tilde \pi_{L},\tilde \pi_{\infty}$ in \eqref{pi}, respectively, with parameters $\theta_i^{(L)}$ in $\Lstat$ and $\tilde \pi_{L}$, and $\theta_{i}=(\beta/(1-\beta))^{i}\lambda$ for $\beta<1/2$ and $\theta_{i}=\lambda$ for $\beta\geq 1/2$
in $\tilde \pi_{\infty}$.
For any sequence $\{\phi_{i}\}$ such that 
\begin{equation*}
\left\{
\begin{array}{ll}
\displaystyle   \sum_{i\geq 1}\phi_i\left(\frac{\beta}{1-\beta}\right)^i\frac{1}{i}<\infty,\quad  &
  \beta< 1/2\\[4pt]
\displaystyle   \sum_{i\geq 1}\frac{\phi_i}{i}<\infty,& \beta\geq 1/2,
\end{array}
\right.
\end{equation*} 
we have
\begin{equation}
\begin{split}  \label{eq:Laplace}
  \E_L&\,\left(\edr^{-\sum_{i\geq1}\phi_i m_i}\right)
  =\frac{\beta\lambda+(1-\beta)\mu_L}
  {\beta\lambda+(1-\beta)\mu_L
  +\sum_{i=1}^{L}\theta^{(L)}_i}
  \tilde\E_L\left(\edr^{-\sum_{i=1}^{L}\phi_i m_i}\right)\\
  &+\frac{1}{\beta\lambda+(1-\beta)\mu_L
  +\sum_{i=1}^L\theta^{(L)}_i}
  \sum_{i=1}^L\tilde\E_L
  \left(\edr^{-\sum_{i=1}^{L}\phi_i m_i}i m_i\right)\\
  =&\,\frac{\beta\lambda+(1-\beta)\mu_L}
  {\beta\lambda+(1-\beta)\mu_L
  +\sum_{i=1}^L\theta^{(L)}_i}
  \tilde\E_L\left(\edr^{-\sum_{i=1}^L\phi_i m_i}\right)\\
  &+\frac{1}{\beta\lambda+(1-\beta)\mu_L
  +\sum_{i=1}^L\theta^{(L)}_i}
  \sum_{i=1}^L\theta_{i}^{(L)}\edr^{-\phi_i}
  \tilde\E_L\left(\edr^{-\sum_{i=1}^L\phi_i m_i}\right)\\
  =&\,\frac{\beta\lambda+(1-\beta)\mu_L
  +\sum_{i=1}^L\theta_i^{(L)}
  \edr^{-\phi_{i}}}{\beta\lambda+(1-\beta)\mu_L
  +\sum_{i=1}^L\theta^{(L)}_i}
  \tilde\E_L\left(\edr^{-\sum_{i=1}^L\phi_i m_i}\right)
\end{split}
\end{equation}
where the second and third equalities follow from 
\begin{equation}\label{laplace derivative}\nonumber
\begin{split}
\E(\edr^{-\phi X})=&\,\exp\{-\theta(1-\edr^{-\phi})\}\\
  \E(X\edr^{-\phi X})=&\,-\frac{\ddr}{\ddr \phi}\E(\edr^{-\phi X})
  =\theta\edr^{-\phi}\exp\{-\theta(1-\edr^{-\phi})\},
\end{split}
\end{equation}
for $X\sim\text{Po}(x; \theta)$.
Here
\begin{equation*}
  \lim_{L\to\infty}\tilde\E_L\left(
  \edr^{-\sum_{i=1}^L\phi_i m_i}\right)
  =\tilde\E_\infty
  \left(\edr^{-\sum_{i\ge1}\phi_i m_i}\right)
\end{equation*} 
as long as
  $$\lim_{L\to\infty}\sum_{i=1}^L\frac{\theta_i^{(L)}}{i}
  (1-\edr^{-\phi_i})
  =\sum_{i\geq1}\frac{\theta_i}{i}(1-\edr^{-\phi_i}).$$
The latter is in turn implied by $\theta_{i}^{(L)}\rightarrow \theta_{i}$ and an application of the monotone convergence theorem, since
  $\theta_i^{(L-1)}\leq \theta_i^{(L)}$
for $L$ large when $\mu_L$ decreases to 0; see Lemma \ref{lemma:monotone} in the Appendix.
As for the first factor on the right hand side of \eqref{eq:Laplace}, an application of Cesaro's Theorem, together with the fact that $\beta\geq 1/2$, $\sum_{i=1}^{L}\theta^{(L)}_{i}\to\infty$ and $\phi_{i}\rightarrow 0$, yields
\begin{equation}\label{eq:limit1}\notag
  \lim_{L\to\infty}
  \frac{\beta\lambda+(1-\beta)\mu_L
  +\sum_{i=1}^{L}\theta_{i}^{(L)}\edr^{-\phi_{i}}}
  {\beta\lambda+(1-\beta)\mu_L
  +\sum_{i=1}^{L}\theta^{(L)}_{i}}
  =1.
\end{equation}   
%
%
%
When $\beta<1/2$ and $\lim_{L\to\infty}L\mu_L=0$,
\begin{equation}\label{eq:limit2}
  \lim_{L\to\infty}
  \frac{\beta\lambda+(1-\beta)\mu_L+\sum_{i=1}^{L}\theta_{i}^{(L)}
  \edr^{-\phi_{i}}}{\beta\lambda+(1-\beta)\mu_L+\sum_{i=1}^{L}\theta^{(L)}_{i}}
  =\frac{\beta\lambda+\sum_{i\geq1}\theta_i
  \edr^{-\phi_{i}}}{\beta\lambda
  +\beta\lambda/(1-2\beta)},
\end{equation}  
where at the numerator we have applied the monotone convergence theorem. Noting that
  $$\frac{\beta\lambda+\sum_{i\geq1}\theta_i
  \edr^{-\phi_{i}}}{\beta\lambda
  + \beta\lambda/(1-2\beta)}
  \tilde\E_\infty
  \left(\edr^{-\sum_{i\ge1}\phi_i m_i}\right)$$
corresponds to the Laplace transform of $m$ under the distribution \eqref{pi in intro} completes the proof.
\end{proof}

Note that, when $\beta\geq 1/2$, the weaker assumption that $\mu_{L}\rightarrow 0$ suffices for the above result. This is informally due to the fact that $\beta\geq 1/2$ makes the addition of size-1 blocks to the system frequent enough to counterbalance the frequency of $L$-sized  blocks entering the system from the right when $\mu_{L}\rightarrow 0$, instead of $L\mu_{L}\rightarrow 0$. When $\beta<1/2$, this is not the case and $\mu_{L}$ must go to zero faster than $1/L$. 

Next, we exploit an embedding of $\MLud$ in $\Mud$ at appropriate stopping times, in order to show that the weak limits in Theorem \ref{prop: limit from dependence 2} describe the long-time behaviour of the B\&D-PU. Note that,  when $\beta\ge1/2$, the result extends Theorem 1 in \citet{ABT92}. 


%

\begin{theorem}\label{prop: limit distribution of delayed Polya urn}
Let $\Mud$ be a \emph{B$\mathrm{\&}$D-PU} with transitions \eqref{BD-PU transitions}, and let  $Z_{j}(h)$ its $j$-th component. Then, as $h\to\infty$, 
  $$(Z_{1}(h),Z_{2}(h),\ldots)
  \stackrel{d}{\to}
  (Z_{1},Z_{2},\ldots)$$ 
where $(Z_{1},Z_{2},\ldots)$:
\begin{enumerate}[{(}i{)}]
\item
have joint distribution \eqref{pi in intro},
for $\beta\in(0,1/2)$;
\item
are independent Poisson random variables with mean $\lambda/i$, for $\beta\in[1/2,1]$.
\end{enumerate}
\end{theorem}

\begin{proof}
Let $\rho_L:\Z_+^\infty\to\Z_+^L$, defined as
\begin{equation}\label{rho}
\rho_{L}(m)=(m_1,\ldots,m_L),
\end{equation} 
be the restriction of $m\in\Z_{+}^{\infty}$ to its first $L$ components. We show that, for any $N\in\N$,
  $\rho_N(\Mud(h))
  \stackrel{d}{\to}
  (Z_{1},\ldots,Z_{N})$
as $h\rightarrow \infty$ in the two regimes. To this end, define the auxiliary chains
$\MLc=\{\MLc(h),h\in\Z_{+}\}$ on $\Z_{+}^{\infty}$, with transition probabilities
\begin{equation}\label{delayed urn modified}
 \tilde p_{L}(m'|m)\propto
  \begin{cases}
\beta\lambda, \quad 
  & m'=m+e_{1},  \\
\beta im_{i}, \quad 
  & m'=m-e_{i}+e_{i+1},\quad i\ge1,  \\
(1-\beta) im_{i}, 
  & m'=m-e_{i}+e_{i-1},\quad i\ne L+1,\\
(1-\beta)\mu_{L}, 
  & m'=m-e_{L+1}+e_{L}.
  \end{cases}
\end{equation} 
and normalising constant $\beta\lambda+(1-\beta)\mu_{L}+\norm{m}-(1-\beta)(L+1)m_{L+1}$.
%
Here $\MLc$ differs from a B\&D-PU in that transitions $m\mapsto m-e_{L+1}+e_{L}$, whereby one item from an $(L+1)$-sized group is removed, have probability proportional to
$(1-\beta)\mu_{L}$ instead of $(1-\beta)(L+1)m_{L+1}$. Note that count $m_L$ and $m_{L+1}$ remain dependent, in view of the transition $m\mapsto m-e_{L}+e_{L+1}$.
Let $\{\mu_{L}\}_{L\geq 1}$ be a decreasing sequence such that, as $L\rightarrow \infty$, $L\mu_L\to 0$.
Let
\begin{equation}\label{markov times 2}
\sigma_{k}
  =\min\big\{h>\sigma_{k-1}:\ \rho_{L}(\MLc(h))\ne \rho_{L}(\MLc(\sigma_{k-1}))\big\},
\end{equation} 
to be the $k$th time a transition of $\MLc$ involves the first $L$ 
counts.
Proposition \ref{prop: L chain embedded 2} in the Appendix shows that $\MLud$ is embedded in $\MLc$ at the stopping times $\sigma_{k}$, in that 
$\{\sigma_{k},n\ge1\}$ occur infinitely often 
and
  $$\P\left(\rho_L(\MLc(\sigma_{k}))
  =m'\left| \rho_L(\MLc(\sigma_{k-1}))=m\right)\right.=
  p_{L}(m'|m),$$
with $p_{L}$ as in \eqref{up-down transitions explicit} with $\mu=\mu_L$. Together with Theorem \ref{thm: up-down chain}, this implies
\begin{equation*}
\rho_L\left(\MLc(\sigma_{k})\right)  \stackrel{d}{\to} 
  (Z_{1}^{(L)},\ldots,Z_{L}^{(L)}),\quad \quad 
  \mbox{as }k\rightarrow \infty,
\end{equation*} 
where the right hand side has distribution $\Lstat$ with $\mu=\mu_L$, see \eqref{invariant up-down}. Clearly, this also implies 
\begin{equation}\label{ergodicity of embedding2}
  \rho_N\left(\MLc(\sigma_{k})\right)
  \stackrel{d}{\to} 
  (Z_{1}^{(L)},\ldots,Z_{N}^{(L)}),\quad 
  \mbox{as }n\rightarrow \infty
\end{equation}
for any $N\leq L$.
Emphasising the dependence on $L$ in \eqref{markov times 2} by $\sigma^{(L)}_{k}$, note now that
 $\{\sigma^{(L)}_{k},n\ge1\}\subset \{\sigma^{(L+1)}_{n},n\ge1\}$
and 
  $\{\sigma^{(L)}_{k},n\ge1\}\ua \N$
as $L\rightarrow \infty$ with probability one, since for all $h\in\N$ there exists an $L_{0}$ such that  $\{1,\ldots,h\}\subset   \{\sigma^{(L)}_{k},n\ge1\}$ for all $L\ge L_{0}$. Therefore, for any given $h$,
  $$\rho_N\left(\MLc(h)\right)=\rho_N\left(\MLc(\sigma_{h})\right)$$
  for $L$ sufficiently large.
The result now follows by taking the limit for $L\rightarrow \infty$ on both sides of \eqref{ergodicity of embedding2}, in virtue of Theorem \ref{prop: limit from dependence 2}.
\end{proof}

The asymptotic regimes of B\&D-PUs are thus determined by the probability of introducing new singleton blocks into the system.
These produce for $\beta\geq1/2$ infinite partitions analogous to those induced by P\'olya urns, since insertion of singletons are frequent enough to make deletions asymptotically irrelevant, and the counts $m_i$ become independent in the limit. 
When $\beta<1/2$, instead, the stream of incoming items is not frequent enough and the dependence between the 
counts $m_i$ is retained in the limit with distribution \eqref{pi in intro}. The next result shows that this latter case provides the reversible and invariant distribution of B\&D-PU with $\beta\in(0,1/2)$. 

\begin{theorem}\label{thm: invariance delayed urn}
Let $M$ have transitions as in \eqref{BD-PU transitions} with $\beta\in(0,1/2)$. Then \eqref{pi in intro} is the reversible and invariant measure of $M$.
\end{theorem}
\begin{proof}
Let $\tilde \pi_{\infty}(m)$ be as in \eqref{pi} 
{and  
  $C^{-1}=
  \beta\lambda+\beta\lambda /(1-2\beta)
  $
be the normalizing constant appearing in \eqref{pi in intro}.} Then, for any $m\in \Z_{+}^{\infty}$,  we have
\begin{align*}
\pi(m+e_{1})&\,p(m\mid m+e_{1})
=C\tilde \pi_{\infty}(m+e_{1})(1-\beta)(m_{1}+1)\\
=&\,C\tilde \pi_{\infty}(m)\frac{\theta_{1}}{m_{1}+1}(1-\beta)(m_{1}+1)
=C\tilde \pi_{\infty}(m)\beta\lambda
=\pi(m)p(m+e_{1}\mid m)
\end{align*}
and for any $i\ge1$
\begin{align*}
\pi(m-e_{i}&\,+e_{i+1})p(m\mid m-e_{i}+e_{i+1})
=C\tilde \pi_{\infty}(m-e_{i}+e_{i+1})(1-\beta)(i+1)(m_{i+1}+1)\\
=&\,C\tilde \pi_{\infty}(m)\frac{im_{i}}{\theta_{i}}\frac{\theta_{i+1}}{(i+1)(m_{i+1}+1)}(1-\beta)(i+1)(m_{i+1}+1)\\
=&\,C\tilde \pi_{\infty}(m)\beta im_{i}
=\pi(m)p(m-e_{i}+e_{i+1}\mid m)
\end{align*}
yielding the result.
Finally, in view of the positive recurrence of the chain, which can be easily proved, $\pi(m)$ is also the unique invariant measure of $M$.
\end{proof}

With a similar argument to that used in Section \ref{sec: up and down chains} for the chains with maximal  allelic count, the invariant distribution  of B\&D-PU's admits representation as augmented vector of independent Poisson variables. Recalling that $\theta_i=(\beta/(1-\beta))^i\lambda$, define $J=j$ with probability proportional to $\theta_{j}$ for $j\ge1$ and proportional to $\beta\lambda$ for $j=0$. 
{One can easily check, using the geometric series $\sum_{i\geq 1}p^i=p/(1-p)$ for $0<p<1$, that}
\begin{equation}\label{eq:J}
  \P(J=0)
  =\frac{1}{2}\bigg(1-\frac{\beta}{1-\beta}\bigg), 
  \quad \quad 
  \P(J=j)
  =\frac{1}{2\beta}\bigg(1-\frac{\beta}{1-\beta}\bigg)
  \bigg(\frac{\beta}{1-\beta}\bigg)^{j},\quad j\ge1.
\end{equation}
This can be obtained from a geometric distribution of parameter $1-\beta/(1-\beta)$, by reallocating half of the mass assigned to 0 to the other support points, resulting in a tilting given by the factor $1/2\beta$.
Next, if $m=(m_1,m_2,\ldots)$ where the $m_{i}$'s are independent each with $\text{Po}(m_{i};\theta_{i}/i)$ distribution,
then 
\begin{equation}\label{eq:aug_Pois}
  m+e_J\sim\pi(m)
\end{equation}  
for $\pi(m)$ in \eqref{pi in intro}.
Note also that in the stationary regime, the random partition induced by B\&D-PUs has a random number of groups $K:=\sum_im_i$,
contrary to the number of groups induced by usual P\'olya urns, recovered here for $\beta\geq1/2$, which grows to infinity asymptotically as $\log h$ (\citealp{KH73}). 
Exploiting the representation \eqref{eq:J}--\eqref{eq:aug_Pois} and using the fact that $\sum_{i\geq 1}p^{i}/i=-\log(1-p)$ for $0<p<1$, one finds that, when $\beta<1/2$,
\begin{equation}\label{moments} K \overset{d}{=}K_\infty+\Indic(J\geq 1),\quad\quad 
  K_\infty\sim
  \mbox{Pois}(k; -\lambda\log(1-\beta/(1-\beta)))
\end{equation}
where $K_\infty$ corresponds to the sum of independent Poisson random variableis with parameters $\theta_i$ as in the l.h.s.~of \eqref{eq:aug_Pois}. 
This immediately yields the moments of $K$, for example
\begin{equation*}
\E(K)=\P(J=0)-\lambda\log(1-\beta/(1-\beta)).
\end{equation*} 


\section{Connection with a mixture of Ewens sampling formulas}
\label{sec: mix of ESFs}

We conclude by showing that the invariant measure of B\&D-PUs corresponds to a mixture of Ewens sampling formulas $\mbox{ESF}_n$ in \eqref{ESF} with a tilted Negative Binomial mixing measure on the sample size $n$. 
Let
\begin{equation}\label{NB pmf}
\mbox{NB}(n;r,p)
=\frac{\Gamma(r+n)}{n!\Gamma(r)}
  p^n(1-p)^r,\quad\quad  n=0,1,\ldots
\end{equation} 
be the Negative Binomial distribution with parameters $r>0$ and $p\in(0,1)$.

\begin{theorem}\label{thm: EFS mixture}
Let $\beta\in(0,1/2)$ and $M$ be a \emph{B\&D-PU} with transition probabilities \eqref{BD-PU transitions} and invariant distribution $\pi$ as in \eqref{pi in intro}. Then 
\begin{enumerate}[(i)]
\item
$N:=\sum_{i\ge1}im_{i}$ is a birth-and-death chain with invariant distribution
\begin{equation*}
  \mu(n)\propto (\beta\lambda+n)\,
  \mathrm{NB}(n;\lambda,\beta/(1-\beta)), \quad \quad n=0,1,\ldots;
\end{equation*} 
\item $\pi$ admits representation as mixture of Ewens sampling formulas
\begin{equation}\label{mixture of ESFs}
\pi(m)=\sum_{n\ge0}\mathrm{ESF}_{n}(\rho_{n}(m))\mu(n), \quad \quad m\in \Z_{+}^{\infty},
\end{equation} 
with $\rho_{n}$ as in \eqref{rho}. 
\end{enumerate}
\end{theorem}


Note that \eqref{mixture of ESFs} is well defined for $n=0$ if one interprets \eqref{ESF} as giving probability one to the empty vector $\rho_{0}(m)=\emptyset$, as one can easily verify that $\pi(0,0,\ldots)=\mu(0)$.

\begin{proof}
From \eqref{BD-PU transitions}, it is easily seen that $N$ is a birth-and-death chain with immigration, whose transitions 
$n\to n\pm1$ have probabilities $t(n\pm 1| n)$ proportional to $\beta(\lambda+n)$ and $(1-\beta)n$ respectively. 
The expected increment of $N$ is proportional to $\beta\lambda-(1-2\beta)n$, which is always positive for $\beta\geq 1/2$, yielding non stationarity. For $\beta<1/2$, using the fact that the expected value of the Negative Binomial distribution in \eqref{NB pmf} is $pr/(1-p)$, it is readily verified that 
\begin{equation}\label{stationary N}
\begin{aligned}
  \mu(n)
  &=\frac{\beta\lambda+n}
  {2\beta\lambda/(1-\beta/(1-\beta))}  
  \mbox{NB}(n;\lambda,\beta/(1-\beta))\\
  &=\frac{\beta\lambda+n}{2\beta\lambda}
  \frac{\lambda_{(n)}}{n!}
  \bigg(\frac{\beta}{1-\beta}\bigg)^{n}
  \bigg(1-\frac{\beta}{1-\beta}\bigg)^{\lambda+1}
\end{aligned}
\end{equation} 
where  $a_{(x)}=a(a+1)\cdots(a+x-1)  =\Gamma(a+x)/\Gamma(a)$ is the Pochhammer symbol.
The detailed balance condition for $N$ then reads
\begin{equation}\label{N reversib}
\begin{aligned}
\mu(n-1)&\,t(n| n-1)
=\frac{\beta\lambda+n-1}{2\beta\lambda}\frac{\lambda_{(n-1)}}{(n-1)!}\bigg(\frac{\beta}{1-\beta}\bigg)^{n-1}\bigg(1-\frac{\beta}{1-\beta}\bigg)^{\lambda+1}\frac{\beta(\lambda+n-1)}{\beta\lambda+n-1}\\
=&\,\frac{\beta\lambda+n}{2\beta\lambda}\frac{\lambda_{(n)}}{n!}\bigg(\frac{\beta}{1-\beta}\bigg)^{n}\bigg(1-\frac{\beta}{1-\beta}\bigg)^{\lambda+1}\frac{(1-\beta)n}{\beta\lambda+n}
=\mu(n)t(n-1| n),
\end{aligned}
\end{equation} 
hence $\mu$ is the reversible measure for $N$, yielding the first assertion. 
To prove \eqref{mixture of ESFs}, it suffices to show that $\pi(m)=\mbox{ESF}_n(\rho_{n}(m)) \mu(n)$
whenever $m\in \Z_{+}^{\infty}$ is such that $\sum_{i\geq 1}i m_i=n$, {for $\text{ESF}_{n}$ and $\mu(n)$ as in \eqref{ESF} and \eqref{stationary N}}. 
Let $p=\beta/(1-\beta)$ and
  $C^{-1}=
  \beta\lambda+\beta\lambda /(1-2\beta)
  =2\beta\lambda/(1-\beta/(1-\beta))
  $
be the {normalising constant} appearing both in \eqref{pi in intro} and in the first display of \eqref{stationary N}.
{Assuming that $\sum\nolimits_{i\geq 1}i m_i=n$}, we have
\begin{equation*}
\begin{aligned}
  \pi(m)
  =&\,C\Big(\beta\lambda+\sum\nolimits_{i=1}^nim_i\Big)
  \prod_{i\ge1}\frac{1}{m_{i}!}
  \bigg(\frac{\lambda p^i}{i}\bigg)^{m_{i}}
  \edr^{-\lambda p^{i}/i}\\
  =&\,C(\beta\lambda+n)
  p^n\exp\bigg\{-\lambda\sum_{i\ge1}\frac{p^i}{i}
  \bigg\}  
  \prod_{i=1}^n
  \bigg(\frac{\lambda}{i}\bigg)^{m_{i}}\frac{1}{m_{i}!}
  \\
  =&\,C(\beta\lambda+n)
  \bigg(\frac{\beta}{1-\beta}\bigg)^{n}
  \bigg(1-\frac{\beta}{1-\beta}\bigg)^{\lambda}\
  \prod_{i=1}^{n}\bigg(\frac{\lambda}{i}\bigg)^{m_{i}}
  \frac{1}{m_{i}!}
\end{aligned}
\end{equation*} 
where in the third equality we have used the fact that $\sum_{i\geq 1}p^{i}/i=-\log(1-p)$. Multiplying and dividing by $\lambda_{(n)}/n!$ now gives the desired condition.
\end{proof}

For what concerns the sampling process associated to B\&D-PUs, say $X$, which alternates sampling of observations from \eqref{Polya urn predictive} to removals of uniformly chosen observations, this evolves in $E=\emptyset\cup (\cup_{n\ge1}\X^{n})$ according to the following transition probabilities
\begin{equation}\label{X transitions}
q(x'|x)\propto
\left\{
\begin{array}{lll}
\beta\lambda, \quad & \text{if } x'=(x_{1},\ldots,x_{N},y),\quad  y\sim P_{0}, \\
\beta, \quad & \text{if }x'=(x_{1},\ldots,x_{j},\ldots,x_{N},x_{j}),\quad  1\le j\le N,\\
(1-\beta), \quad & \text{if }x'=(x_{1},\ldots,x_{j-1},x_{j+1},\ldots,x_{N}), \quad 1\le j\le N,
\end{array}
\right.
\end{equation} 
with normalising constant $\beta\lambda+N$. 
It is immediate from Theorem \ref{thm: EFS mixture} to see that this particle process is also stationary when $\beta<1/2$ with invariant measure given by the mixture of P\'olya urn schemes
\begin{equation}\label{stationary for N}\nonumber
\text{PU}(\d x):=\sum_{n\ge0}\text{PU}_{n}(\d x)\mu(n),
\end{equation} 
where $\mu$ is as in Theorem \ref{thm: EFS mixture}, $\text{PU}_{n}(\d x)$ represents the joint law of $X_{1},\ldots,X_{N}$ drawn from the P\'olya urn scheme \eqref{Polya urn predictive}, conditional on $N=n$, and $\text{PU}_{0}$ assigns probability one to the empty set.

\section*{Acknowledgements}

The first two authors are supported by the European Research Council (ERC) through StG ``N-BNP'' 306406.

\section*{Appendix}\label{sec: proofs}
\phantomsection\addcontentsline{toc}{section}{Appendix}

\renewcommand{\thesection}{\Alph{section}}
\setcounter{section}{1}
\setcounter{theorem}{0}

\begin{lemma}\label{lem: system}
Let $\lambda,\mu>0$. Then the system
\begin{equation*}
\begin{cases}
  \theta_{1}-(1-\beta)\theta_{2}= \beta\lambda \\
  \theta_{i}-\beta\theta_{i-1}-(1-\beta)\theta_{i+1}=0,\quad i=2,\ldots,L-1,   \\
  \theta_{L}-\beta\theta_{L-1}=(1-\beta)\mu.
  \end{cases}
\end{equation*} 
has solution 
\begin{equation}\label{eq: mixed thetas}
  \theta_{i}
  =\begin{cases}
  \displaystyle
  \beta^i\frac{(1-\beta)^{L-i+1}-\beta^{L-i+1}}{(1-\beta)^{L+1}-\beta^{L+1}}\lambda
  +(1-\beta)^{L-i+1}\frac{(1-\beta)^i-\beta^i}{(1-\beta)^{L+1}-\beta^{L+1}}\mu,
  &\beta\neq 1/2,\\[3mm]
  \displaystyle
  \frac{L-i+1}{L+1}\lambda+\frac{i}{L+1}\mu,
  &\beta=1/2.
  \end{cases}
\end{equation}  
Moreover
\begin{equation}\label{eq:sum_thetaL}
  \sum_{i=1}^L \theta_i
  =\begin{cases}
\displaystyle   \frac{(1-\beta)^{L+1}\mu-\beta^{L+1}\lambda}{(1-\beta)^{L+1}-\beta^{L+1}}L+
  \frac{(1-\beta)\beta}{1-2\beta}\frac{(1-\beta)^L-\beta^L}
  {{(1-\beta)^{L+1}-\beta^{L+1}}}(\lambda-\mu),&
  \beta\neq 1/2,\\[3mm]
\displaystyle   \frac{\lambda+\mu}{2}L,&
  \beta=1/2.
  \end{cases}
\end{equation}  
\end{lemma}
\begin{proof}
The system can be written $A\, {\bm x}={\bs\beta}$, 
where ${\bs\beta}=\beta\lambda e_1+(1-\beta)\mu e_L$ and $A$ is a tridiagonal matrix with entries $-(1-\beta),1,-\beta$ respectively above, on and below the main diagonal. 
Let $d_L=\det(A)$ where subscript $L$ refers to the dimension of $A$. Then $d_l$, $1\leq l\leq L$, corresponds to the determinant of the $l\times l$ submatrix made of the first $l$ rows and columns of $A$, i.e. $det(A)$ when $L=l$.
By using expansion by the first column, it is easy to check that $d_l$ satisfies the recursive relation
\begin{equation}\label{eq:recursive}
  d_l=d_{l-1}-\beta(1-\beta)d_{l-2},\quad l\geq 2
\end{equation}  
with initial values $d_1=1,d_2=1-\beta(1-\beta)$. Consider the associated second order difference equation
\begin{equation}\label{eq:diff_equ}
  x_{t+2}-x_{t+1}+\beta(1-\beta)x_t,\quad t\geq 2
\end{equation}
i.e. \eqref{eq:recursive} with $d_t=x_{t-1}$. It has characteristic equation
  $m^2-m+\beta(1-\beta)=0$
with roots $m_1=b$ and $m_2=(1-\beta)$. When $\beta\neq 1/2$, \eqref{eq:diff_equ} has general solution given by
  $$x_t=c_1 \beta^t+c_2(1-\beta)^t,\quad c_1,c_2\in\R$$
so that by using the initial conditions $x_0=0$ and $x_1=1-\beta(1-\beta)$, 
we find
  $c_2=(1-\beta)^2/(1-2\beta),\ c_1=1-c_2=-\beta^2/(1-2\beta)$
and, in turn,
\begin{equation}\label{eq:d_L}
  d_l=((1-\beta)^{l+1}-\beta^{l+1})/(1-2\beta)
\end{equation}  
When $\beta=1/2$, \eqref{eq:diff_equ} has general solution given by  $x_t=(c_1+c_2t)(1/2)^{t},\ c_1,c_2\in\R$
where $1/2$ is the common root of the characteristic equation. By using the initial conditions $x_0=0$ and $x_1=1-\beta(1-\beta)=1-1/4$, 
we find
  $c_1=1;\ c_2=1/2$,
therefore
\begin{equation}\label{eq:d_L2}
  d_l=\left(1+\frac{l-1}{2}\right)2^{-(l-1)}
  =(l+1)2^{-l}.
\end{equation}  
Note that $d_l\neq0$ for any $l$ since the constant solution of \eqref{eq:diff_equ}, $x_t=0$, is ruled out by the initial condition $x_0=1$. Moreover, $d_{l}>\beta(1-\beta)d_{l-1}$, so that $d_{l}>0$ for any $l$, cfr. \eqref{eq:recursive}.
In particular,
  $d_l>\beta^{l-1}(1-\beta)^{l-1},\ l\geq 2$.
Since for $\beta\neq 0,1$ the absolute value of the roots of the characteristic equation is $<1$, the constant solution $x_t=0$ is stable, meaning that $x_t\to 0$ as $t\to \infty$, i.e.  
  $\lim_{l\to\infty}d_l= 0$. When $\beta=0,1$, \eqref{eq:diff_equ} has constant solution $x_t=1$, hence $d_l=1$ for any $l$.

Since now $d_L\neq 0$, the solution is unique and given by 
  $$\bm\theta=A^{-1}{\bs\beta}
  =\beta A^{-1}e_1\lambda+(1-\beta)A^{-1}e_L\mu$$
and, in particular, 
\begin{equation}\label{eq:thetai}
  \theta_i=\beta a_{i1}\lambda+(1-\beta)a_{iL}\mu,\quad
  i=1,\ldots,L
\end{equation}
where $a_{ij}$ is the $(ij)$th entry of $A^{-1}$. By using Cramer's method,
   $$A^{-1}=\frac{1}{\det A}\left( (-1)^{i+j}\det A_{ij} \right)^T$$
where $A_{ij}$ is the $(L-1)\times(L-1)$ matrix obtained by deletion of the $i$th row and $j$th column of $A$. 
 Hence,
  $$a_{i1}=d_L^{-1}(-1)^{1+i}\det A_{1i},\quad
  a_{iL}=d_L^{-1}(-1)^{L+i}\det A_{Li}$$
Consider $a_{iL}$ first. It is easy to see that $A_{L1}$ is lower triangular with $-(1-\beta)$ in the diagonal, hence $\det(A_{L1})=[-(1-\beta)]^{L-1}$. For $1<i<L$, $A_{Li}$ is a $(L-1)\times(L-1)$ matrix that can be partitioned as
  $$A_{Li}=\begin{bmatrix}
  B_{11}&B_{12}\\B_{21}&B_{22}\end{bmatrix}$$
where $B_{11}$ correspond to $A$ with $i-1$ rows, $B_{12}$ is a $(i-1)\times(L-i)$ made of all zero and $B_{22}$ is a lower triangular matrix with $L-i$ rows and $-(1-\beta)$ in the diagonal. In particular, $\det(B_{11})=d_{i-1}$ and $\det(B_{22})=[-(1-\beta)]^{L-i}$. By using the formula for the determinant of partitioned matrices,
  $$\det\begin{pmatrix}\begin{bmatrix}
  B_{11}&B_{12}\\B_{21}&B_{22}\end{bmatrix}\end{pmatrix}
  =\det(B_{11})\det(B_{22}-B_{21}B_{11}^{-1}B_{12})$$
we find that
  $$\det A_{Li}=\det(B_{11})\det(B_{22})
  =d_{i-1}(-1)^{L-i}(1-\beta)^{L-i}$$
so that
\begin{align*}
  a_{iL}
=d_L^{-1}(-1)^{L+i}d_{i-1}(-1)^{L-i}(1-\beta)^{L-i}
=d_L^{-1}(-1)^{2L}d_{i-1}(1-\beta)^{L-i}
  =d_L^{-1}d_{i-1}(1-\beta)^{L-i}
\end{align*}
Finally, $A_{LL}$ corresponds to $A$ with $L-1$ rows, hence $\det(A_{LL})=d_{L-1}$. Summing up, by using the convention $d_0=1$, 
\begin{equation}\label{eq:aiL}
  a_{iL}=d_L^{-1}(1-\beta)^{L-i}d_{i-1},\quad
  i=1,\ldots,L
\end{equation}  
As for $a_{i1}$, we can exploit a certain symmetry of $A$: using the notation $A=A_{\beta}$ ,
  $\left[A_{\beta}\right]^T=A_{1-\beta}$,
so that 
  $A_\beta^{-1}=(A_{1-\beta}^{-1})^T$  
to find that
\begin{equation}\label{eq:ai1}
  a_{i1}=d_L^{-1}\beta^{i-1}d_{L-i},\quad
  i=1,\ldots,L.
\end{equation}  
Plugging in \eqref{eq:aiL} and \eqref{eq:ai1} into \eqref{eq:thetai}, one obtains
\begin{align*}
  \theta_{i}
  &=\frac{\beta^{i}d_{L-i}}{d_{L}}\lambda
  +\frac{(1-\beta)^{L-i+1}d_{i-1}}{d_{L}}\mu\end{align*} 
and the thesis follows by using \eqref{eq:d_L} and \eqref{eq:d_L2}. By direct calculation, one can check that $\theta_i$ is a convex linear combination of $\lambda$ and $\mu$. Alternatively, one can use directly equation \eqref{eq:thetai}. Let  ${\bm 1}=e_1+\ldots+e_L=(1,\ldots,1)$. Since
  $A{\bm 1}=(\beta,0,\dots,0,(1-\beta))$, we have
  $\beta A^{-1}e_1+(1-\beta)A^{-1}e_L={\bm 1}$,
that is
  $\beta a_{i1}+(1-\beta)a_{iL}=1$
for any $1\leq i\leq L$. 
Finally, $\beta a_{i1},(1-\beta)a_{iL}\geq 0$ as a simple calculation reveals. This completes the proof.

As for \eqref{eq:sum_thetaL}, the result for $\beta=1/2$ is straightforward. When $\beta\neq 1/2$, it is convenient to write 
  $$\theta_{i}^{(L)}=
  \frac{(1-\beta)^{L+1}\left(\frac{\beta}{1-\beta}\right)^i
  -\beta^{L+1}}{(1-\beta)^{L+1}-\beta^{L+1}}\lambda
  +\frac{(1-\beta)^{L+1}}
  {(1-\beta)^{L+1}-\beta^{L+1}}\left[1-\left(
  \frac{\beta}{1-\beta}\right)^i\right]\mu$$
We have
\begin{align*}
  \sum_{i=1}^L \theta_i^{(L)}
  &=\frac{1}{(1-\beta)^{L+1}-\beta^{L+1}}
  \left\{\left[(1-\beta)^{L+1}\sum_{i=1}^L\left(\frac{\beta}{1-\beta}\right)^i
  -L\beta^{L+1}\right]\lambda\right.\\
  &\left.+\left[L(1-\beta)^{L+1}-(1-\beta)^{L+1}\sum_{i=1}^L\left(\frac{\beta}{1-\beta}\right)^i
  \right]\mu\right\}\\
  &=\frac{(1-\beta)^{L+1}\mu-\beta^{L+1}\lambda}{(1-\beta)^{L+1}-\beta^{L+1}}L
  +\frac{(1-\beta)^{L+1}}{(1-\beta)^{L+1}-\beta^{L+1}}
  \sum_{i=1}^L\left(\frac{\beta}{1-\beta}\right)^i(\lambda-\mu)
\end{align*}
The result follows by the formula of the sum of the first $L$ terms of a geometric series,
  $$\sum_{i=1}^L\left(\frac{\beta}{1-\beta}\right)^i
  =\frac{\beta}{1-2\beta}\frac{(1-\beta)^L-\beta^L}{(1-\beta)^L}.$$
  The simplification \eqref{theta central} can be derived from \eqref{eq: mixed thetas} by using $x^2-y^2=(x-y)(x+y)$.
\end{proof}

\begin{lemma}\label{lemma:monotone}
Let $\theta_{i}^{(L)}$ be defined as in  \eqref{eq: mixed thetas}
with $\mu=\mu_L$ and $\{\mu_L\}_{L\geq 1}$ a decreasing sequence of positive real numbers such that $\mu_L\to 0$ as $L\to\infty$. Then
  $\theta_i^{(L)}\geq\theta_i^{L-1}$
for $L$ large enough and any $i=1,\ldots,L$.
\end{lemma}
\begin{proof}
Consider the case of $\beta\neq 1/2$ and, as a short hand notation, let $p=\beta/(1-\beta)$ and $c=[(1-\beta)^{L+1}-\beta^{L+1}][(1-\beta)^{L}-\beta^L]$. We have
\begin{align*}
  \theta_i^{(L)}-\theta_i^{(L-1)}
  =&\,\left\{
  \frac{(1-\beta)^{L+1}p^i-\beta^{L+1}}{(1-\beta)^{L+1}-\beta^{L+1}}
  -\frac{(1-\beta)^{L}p^i-\beta^{L}}{(1-\beta)^{L}-\beta^{L}}
  \right\}\lambda\\
  &+\left\{
  \frac{(1-\beta)^{L+1}}{(1-\beta)^{L+1}-\beta^{L+1}}\mu_L
  -\frac{(1-\beta)^{L}}{(1-\beta)^{L}-\beta^{L}}\mu_{L-1}
  \right\}(1-p^i)\\
  =&\,\frac{(1-\beta)^L\beta^L}{c}
  \left\{-(1-\beta)p^i-\beta+1-\beta+\beta p^i
  \right\}\lambda\\
  &+\frac{1}{c}
  \left\{
  (1-\beta)^{2L+1}(\mu_L-\mu_{L-1})
  -(1-\beta)^L\beta^L\left[(1-\beta)\mu_L-\beta \mu_{L-1}\right]
  \right\}(1-p^i)
  \\
  =&\,\frac{(1-\beta)^L\beta^L}{c}
  (1-2\beta)(1-p^i)\lambda
  +\frac{1}{c}
  \left\{
  (1-\beta)^{2L+1}(\mu_L-\mu_{L-1})
  \phantom{\frac{1}{2}}\right.\\
  &\left.\qquad\qquad
  -(1-\beta)^L\beta^L(1-2\beta)\left[\mu_L+\frac{\beta}{1-2\beta}(\mu_L-\mu_{L-1} \right]
  \right\}(1-p^i)
  \\
  =&\,\frac{(1-\beta)^L\beta^L(1-p^i)}
  {c}
  \left\{
  (1-2\beta)(\lambda-\mu_L)+\left[
  (1-\beta)p^{-L}-\beta \right](\mu_L-\mu_{L-1})
  \right\}.
\end{align*}
Note that $c\geq0$ for any $\beta\neq1/2$. When $\beta<1/2$, $1-p^i\geq0$, $(1-2\beta)\geq 0$ and $\lambda\geq \mu_L$ for $L$ large since $\mu_L\to 0$. As for the second term $\left[(1-\beta)p^{-L}-\beta\right](\mu_L-\mu_{L-1})$ in curly brackets,  $(1-\beta)p^{-L}-\beta\uparrow\infty$ and $\mu_L-\mu_{L-1}\geq 0$ as $L$ gets large when $\mu_L$ is decreasing. When $\beta>1/2$, $1-p^i\leq0$, $(1-2\beta)\leq 0$ and $\lambda\geq \mu_L$ for $L$ large since $\mu_L\to 0$. Also $(1-\beta)p^{-L}-\beta\downarrow -\beta$ and $\mu_L-\mu_{L-1}\geq 0$ given the monotonicity of $\{\mu_L\}_{L\geq 1}$. So also in this case $\theta_i^{(L)}-\theta_i^{(L-1)}\geq0$.
When $\beta=1/2$,
\begin{align*}
  \theta_i^{(L)}-\theta_i^{(L-1)}
  &=\left\{\frac{L-i+1}{L+1}-\frac{L-i}{L}\right\}\lambda
  +\left\{\frac{i}{L+1}\mu_L-\frac{i}{L}\mu_{L-1}\right\}\\
  &=\frac{1}{L(L+1)}\left[i(\lambda-\mu_{L-1})+iL(\mu_L-\mu_{L-1})\right]
  \geq 0
\end{align*}
for $L$ large since $\lambda\geq \mu_{L-1}$ and $\mu_L-\mu_{L-1}\geq0$.
\end{proof}

\begin{proposition}\label{prop: L chain embedded 2}
Let $\MLc$ have transitions \eqref{delayed urn modified}, and $\MLud$ as in Definition \eqref{up-down transitions explicit} with $\mu=\mu_{L}$. For any $L<\infty$, $\MLud$ is embedded in $\MLc$ at the Markov times $\sigma_{k}$ in the sense that the times $\{\sigma_{k},n\ge1\}$ occur infinitely often in $\N$  and
\begin{equation*}
  \P\Big(\rho_L\Big(\MLc(\sigma_{k})\Big)
  =m'\Big| \rho_L\Big(\MLc(\sigma_{k-1})\Big)=m\Big)=
  p_{L}(m'|m).
\end{equation*} 
for all $k\in\N$ and $m,m'\in\Z_+^L$.
\end{proposition}
\begin{proof}
Let $L_i$ (resp.~$R_i$) denote the event that the $i$th transition after $\sigma_{k-1}$ involves components $m_{1},\ldots,m_{L}$ (resp.~$m_{L+1},m_{L+2},\ldots$).
Denote also by $\P_{k-1,m}(\cdot)$ the conditional probability $\P(\cdot|M(\sigma_{k-1})=m)$ and $\E_{k-1,m}(\cdot)$ for the respective conditional expectation.
Given $m\in\Z_{+}^{\infty}$, 
we have
\begin{align*}
  \P_{k-1,m}(&\,\sigma_{k}=\sigma_{k-1}+1)
  =\P_{k-1,m}(L_1)
  =\frac{\eta_L +\sum_{i=1}^{L}im_{i}}
  {\eta_L +\norm{m}-(1-\beta)(L+1)m_{L+1}}
\end{align*}
where $\eta_{L}=\beta\lambda+(1-\beta)\mu_{L}$ and, for $h\geq 2$,
  $$\P_{k-1,m}(\,\sigma_{k}=\sigma_{k-1}+h)=
  \P_{k-1,m}(L_h|R_1,\ldots,R_{h-1})
  \prod_{\ell=1}^{h-1}\P_{k-1,m}(R_\ell|R_1,\ldots,R_{\ell-1}).$$
The denominator of $\P_{k-1,m}(L_h|R_1,\ldots,R_{h-1})$ depends on the intermediate transitions and is thus random. We can factorise
  $$1-\sum_{h=1}^{H}\P_{k-1,m}(\sigma_{k}=\sigma_{k-1}+h)
  =\prod_{h=1}^{H}\left[1-\P_{k-1,m}(L_h|R_1,\ldots,R_{h-1})\right]$$
and the set $\{\sigma_n,n\geq 1\}$ has an infinite number of terms as long as
  $$\sum_{h\geq 1}\P_{k-1,m}(L_h|R_1,\ldots,R_{h-1})\to\infty.$$
The latter holds since 
  $$\P_{k-1,m}(L_h|R_1,\ldots,R_{h-1})\geq 
  \frac{\eta_L+\sum_{i=1}^{L}im_{i}}
  {\eta_L+\norm{m}+h-1}.$$
Since now the transition at step $h$ is the first to involve the first $L$ components, whose configuration has not changed from $\rho_{L}(M(\sigma_{k-1}))$, the numerator of the respective probability can be isolated to write
\begin{align*}
  \P_{k-1,m}(&\,\sigma_{k}=\sigma_{k-1}+h)\\
  =&\,\left(\eta_L+\sum\nolimits_{i=1}^{L}im_{i}\right)
  \E_{|n,m}(T_h^{-1}|R_1,\ldots,R_{h-1})
  \prod_{\ell=1}^{h-1}\P_{k-1,m}(R_\ell|R_1,\ldots,R_{\ell-1})
\end{align*}
where 
  $$T_h
  =\eta_L +\sum_{i\geq 1}iM_{i}(\sigma_{k-1}+h-1)-\beta(L+1)M_{L+1}(\sigma_{k-1}+h-1)$$
denotes the denominator in 
$\P_{k-1,m}(L_{h})$ and $M_{i}(h)$ the random variable for the $i$th component of $M(h)$.
A similar derivation for the transition $m-e_{i}+e_{i+1}$, $i=1,\ldots,L$, occurring at step  $h$ after $\sigma_{k-1}$, leads to writing
\begin{align*}
  \P_{k-1,m}(\rho_L(M&\,(\sigma_{k-1}+h))=\rho_L(m-e_{i}+e_{i+1}),
  \sigma_{k}=\sigma_{k-1}+h)=\\
  =&\,
  \beta im_{i}\
  \E_{k-1,m}(T_h^{-1}|R_1,\ldots,R_{h-1})
  \prod_{\ell=1}^{h-1}\P_{k-1,m}(R_\ell|R_1,\ldots,R_{\ell-1}),
\end{align*}
from which, in turn,
\begin{align*}
    \P\left(\rho_L(M(\sigma_{k}))=\rho_L(m-e_{i}+e_{i+1})
  \mid M(\sigma_{k-1})=m,\sigma_{k}=\sigma_{k-1}+h\right)=\frac{\beta im_{i}}{\eta+\sum_{i=1}^{L}im_{i}}.
\end{align*}
An analogous statement can be derived with a similar argument for all other transitions involving one of the first $L$ components, which, in view of the independence on $h$ of the right hand side, leads to the result. 
\end{proof}

\end{document}